\documentclass[review,onefignum,onetabnum]{amsart}

\usepackage{cleveref}
\usepackage{cite}
\usepackage{url}

\newtheorem{proposition}{Proposition}
\newtheorem{lemma}{Lemma}
\theoremstyle{definition}
\newtheorem*{remark}{Remark}
\newtheorem{step}{\textsc{Step}}
\crefname{step}{Step}{Step}
\newtheorem{experiment}{Experiment}
\crefname{experiment}{Experiment}{Experiment}
\crefname{figure}{Figure}{Figure}

\title[Numerical reconstruction of radiative sources]{Numerical reconstruction of radiative sources in an absorbing and
non-diffusing scattering medium in two dimensions}

\author[H.~Fujiwara]{Hiroshi Fujiwara}
\address{Graduate School of Informatics,  Kyoto University, Yoshida Honmachi, Sakyo-ku, Kyoto 606-8501, Japan}
\email{fujiwara@acs.i.kyoto-u.ac.jp}

\author[K.~Sadiq]{Kamran Sadiq}
\address{Johann Radon Institute of Computational and Applied Mathematics (RICAM), Altenbergerstrasse 69, 4040 Linz, Austria}
\email{kamran.sadiq@ricam.oeaw.ac.at}

\author[A.~Tamasan]{Alexandru Tamasan}
\address{Department of Mathematics, University of Central Florida, Orlando, 32816 Florida, USA}
\email{tamasan@math.ucf.edu}

\usepackage{braket,amsopn,amssymb}
\DeclareMathSymbol{\Real}{\mathalpha}{AMSb}{"52}
\DeclareMathSymbol{\C}{\mathalpha}{AMSb}{"43}
\DeclareMathSymbol{\Z}{\mathalpha}{AMSb}{'132}
\DeclareMathOperator{\Repart}{Re}
\DeclareMathOperator{\Impart}{Im}

\DeclareMathOperator{\Arg}{Arg}
\DeclareMathOperator{\pv}{p\text{.}v\text{.}}
\DeclareMathOperator{\Div}{\boldsymbol{D}}
\DeclareMathOperator{\Hilb}{\boldsymbol{H}}
\DeclareMathOperator{\Radn}{\boldsymbol{R}}

\DeclareMathOperator{\supp}{supp}
\providecommand{\norm}[1]{\left\lVert#1\right\rVert}
\newcommand{\mut}{\mu_\text{t}}
\newcommand{\mua}{\mu_\text{a}}
\newcommand{\mus}{\mu_\text{s}}
\newcommand{\D}{D}
\newcommand{\xip}{\xi^{\perp}}
\def\sep{\:;\:}

\usepackage{tikz}



\begin{document}

\maketitle

\begin{abstract}
We consider the two dimensional quantitative imaging problem of
recovering a radiative source inside an absorbing and scattering
medium from knowledge of the outgoing radiation measured at the boundary. 
The medium has an anisotropic scattering property
that is neither negligible nor large enough for the diffusion approximation
to hold. We present the numerical realization of the authors' recently
proposed reconstruction method.
For scattering kernels of finite Fourier content
in the angular variable, the solution is exact. 
The feasibility of the proposed algorithms is demonstrated in several
numerical experiments, including simulated scenarios for parameters
meaningful in optical molecular imaging.
\end{abstract}

\section{Introduction}
We consider the inverse source problem for radiative transport in a bounded,
strictly convex domain $D \subset \Real^2$ with boundary $\Gamma$,
as modeled by the linearized Boltzmann equation. Let $S^1 = \set{\xi \in \Real^2 \sep |\xi| = 1}$ be the unit circle,
and $\Gamma_\pm:=\set{(x,\xi)\in \Gamma \times S^1 \sep \pm\nu(x)\cdot\xi>0 }$ be the incoming ($-$),
respectively outgoing ($+$), unit tangent sub-bundles of the boundary,
where $\nu(x)$ is the outer unit normal at $x \in \Gamma$.
  
The medium is characterized by the absorption coefficient $\mua(x)$, scattering coefficient $\mus(x)$,
and scattering kernel $p(x,\xi\cdot\xi')$, all of which are assumed known, real valued, non-negative,
essentially bounded functions. For any $x\in D$ and $\xi,\xi'\in S^1$, the scattering kernel (a conditional probability)
satisfies $\displaystyle\int_{S^1}p(x,\xi\cdot\xi')d\sigma_{\xi'} =1$,  for all $x\in D$. The total attenuation is defined as $\mut = \mua+ \mus$. 

Generated by an unknown source $q$, in the steady state case, and assuming no incoming radiation from outside the domain,
the density of particles $I(x,\xi)$ at $x\in D$ traveling in the direction $\xi$ satisfies the problem
\begin{subequations} \label{eq:forward}
\begin{gather}
\begin{split}
&\xi\cdot\nabla_x I(x,\xi) +\mut(x) I(x,\xi)  \\
&\qquad = \mus(x) \int_{S^1} p(x,\xi\cdot\xi')I(x,\xi') d\sigma_{\xi'} + q(x) , \quad (x,\xi)\in D\times S^1, \label{TransportScatEq1} \\
\end{split}
\\
I(x,\xi) \bigm|_{\Gamma_{-}} = 0, \label{no_incoming}
\end{gather} 
\end{subequations}
The unknown source $q$ is assumed square integrable and compactly supported in $D$.
Under the above assumptions on $\mu_a,\mu_s, p$, and $q$,
the forward problem \eqref{eq:forward} is well posed (\cite{dautrayLions4}),
with a unique solution $I$ in the space
\[\set{f\in L^2(D\times S^1) \sep (x,\xi)\mapsto \xi\cdot\nabla f(x,\xi)\in L^2(D\times S^1)}.\]
For well-posedness results under various other assumptions,
see \cite{choulliStefanov96, choulliStefanov99, anikonov02,mokhtar},
and the generic result in \cite{stefanovUhlmann08}. 

For a given medium, i.e., $\mua$, $\mus$ and $p$ are known, we consider the following inverse source problem:
Determine $q$ in $D$ from the measurement 
\[
I(x,\xi) = I_{\text{measure}}(x,\xi), \qquad (x,\xi)\in \Gamma_{+}
\]
of the directional outflow at the boundary. In the presence of scattering, this problem is
equivalent to inverting a smoothing operator, and therefore it is ill-posed.

When $\mua=\mus=0$, this is the classical $X$-ray tomography problem of Radon \cite{radon1917},
where $q$ is to be recovered from its integrals along lines, a problem which is well understood
both on its theoretical and numerical facets, see, e.g., \cite{nattererBook,helgasonBook,kuchmentBook}
and references therein. For $\mua\neq 0$ but $\mus=0$, this is the problem of inversion of
the Attenuated $X$-ray transform in two dimensions, solved by different methods in \cite{ABK}, and \cite{novikov01};
see \cite{natterer01,bomanStromberg,bal04,monard16, monard17,kunyansky01} for later approaches and some numerical implementation. 

The inverse source problem in a scattering media considered here, i.e., $ \mus p\neq 0$,
has also been considered under various limiting constraints (e.g., \cite{larsen75, siewert93, mccormickSanchez, hubenthal,balTamasan07})
with a most general result showing that the source is uniquely and stably determined by the outflow in \cite{stefanovUhlmann08}.
However, the numerical solutions for the inverse source problem based on the
above mentioned results have yet to be realized. 

In the special case of a weakly (anisotropic) scattering media, where $\|1- p\|\ll 1$,
one may recover the source by devising algorithms for the iterative method proposed in \cite{balTamasan07}.
However, based on a perturbation argument of the non-scattering case in \cite{novikov01},
the method does not extend to strongly anisotropic scattering media considered here.
Moreover, even in the case of a weakly scattering media, the requirement of solving one forward problem
(a computationally extensive procedure) at each iteration renders the method inefficient from an imaging perspective.

In here we present a numerical reconstruction method based on the authors' recent theoretical results
in \cite{FujiwaraSadiqTamasan}, and propose one algorithm to recover the radiative sources.
As demonstrated in the numerical experiments below, the algorithms can handle the quantitative imaging
of sources in a non-small scattering medium that is far from diffusive approximation, with applications to
Optical Molecular Imaging \cite{NW_OMI, NYS_OMI}.

Key to the reconstruction method is the realization that any finite Fourier content in the angular variable
of the scattering kernel splits the problem into a non-scattering one and a boundary value problem
for a finite elliptic system. The role of the finite Fourier content has been independently recognized in \cite{monardBalPreprint}. However, in general, the scattering kernel does have an infinite Fourier content.
Such is the case in the numerical experiments below, where we work with the ubiquitous
(two dimensional version of) the Henyey-Greenstein kernel
\begin{equation}\label{HG}
p(x,\xi\cdot\xi')=p(\xi\cdot\xi')
= \dfrac{1}{2\pi} \dfrac{1-g^2}{1-2g \xi\cdot\xi' + g^2},\quad\forall x\in D.
\end{equation}
In \eqref{HG} the parameter $0\leq g\leq 1$ models a degree of anisotropy
with $g=0$ being the isotropic case, and $g=1$ being the ballistic case.
In the proposed reconstruction we use an approximate scattering kernel obtained by truncating
the Fourier modes in the angular variable. The error estimate in the data due to such truncation
(see \Cref{sec:procedure} below) allows to interpret the reconstruction as a minimum residual solution,
where for a given a priori noise level, the degree of truncation dictates the number of significant Fourier modes.
Moreover, we devise a locally optimal criterion for the choice of the order of this truncation, which is independent of the unknown source.

The theoretical method originated in Bukhgeim's theory of $A$-analytic functions developed
in \cite{bukhgeimBook} to treat the non-attenuating case and in \cite{ABK} for the absorbing but non-scattering case,
and extends the ideas in \cite{sadiqTamasan01,sadiqTamasan02} to the scattering case.
One of the numerical results in the first experiment below has been announced in \cite{FujiwaraSadiqTamasan}.


\section{Preliminaries}\label{Sec:prelim}
In this section we establish notation, while presenting the basic ideas of the reconstruction in the non-scattering case.
The presentation follows authors' ideas in \cite{sadiqTamasan01,sadiqTamasan02,FujiwaraSadiqTamasan},
while the notation is that used by practitioners in Optical Tomography, e.g. \cite{arridge}.
We also depart from the original notation in \cite{bukhgeimBook} (used so far in \cite{sadiqTamasan01,sadiqTamasan02,FujiwaraSadiqTamasan}),
and work with the positive Fourier modes. This will allow for the natural indexing of sequences.
For the  analytical framework, which specifies the regularity of the coefficients
and proves the appropriate convergence of the ensuing series we refer to \cite{FujiwaraSadiqTamasan}.
We identify points $x = (x_1, x_2) \in \Real^2$ with their complex representative $z=x_1+ix_2 \in \C$.
The method considers the transport model  \eqref{eq:forward} in the Fourier domain of the angular variable.
For $\xi = \xi(\theta) = (\cos \theta,\sin\theta)\in S^1$, let
\begin{equation}\label{fourierI}
I(z,\xi) = \sum_{m \in \Z} I_{m}(z) e^{-im\theta}
\end{equation}
be the Fourier series representation of $I$ in the angular variable.
From the analysis of forward problem (e.g., \cite{dautrayLions4}),
for a square integrable source and essentially bounded coefficients,
the solution $I$ of \eqref{eq:forward} is at least square integrable
in the angular variable and thus the series \eqref{fourierI} is summable in $L^2$ sense.
Moreover, since $I$ is real valued, $I_{-m}=\overline{I_m}$ and the angular dependence
is completely determined by the sequence of its nonnegative Fourier modes 
\[
D \ni z\mapsto \langle I_{0}(z), I_{1}(z),I_{2}(z),... \rangle.
\]
Similarly, let 
\[
p(z,\xi\cdot\xi') = \sum_{m \in \Z} p_{m}(z) e^{-im\tau}
\]
be the Fourier series representation of the scattering kernel,
where $\tau$ is the angle formed by $\xi$ and $\xi'\in S^1$.
We assume that $\theta\mapsto p(z,\cos(\theta))$ is uniformly in $z\in D$ sufficiently smooth,
such that its Fourier coefficients $p_n$ have sufficient decay in $n$
for the convergence analysis of ensuing series; see such details
in \cite{sadiqTamasan01,FujiwaraSadiqTamasan}.
Moreover, since $p(z,\cos\tau)$ is both real valued and even in $\tau$,
$p_{m}$ are real valued and $p_m=p_{-m}$, for $m \in \Z$. 

By introducing the Cauchy-Riemann gradients in the spatial domain
$\partial = (\partial_{x_1} - i\partial_{x_2})/2$ and
$\overline{\partial} = (\partial_{x_1} + i\partial_{x_2})/2$,
the advection operator becomes $\xi \cdot\nabla_x=e^{-i\theta}\overline{\partial} + e^{i\theta}\partial$.
A projection on the basis $\set{e^{im\tau}}$ in $L^2(0,2\pi)$,
reduces the original transport equation \eqref{TransportScatEq1} to the infinite dimensional elliptic system
\[
\overline{\partial} I_{-1}(z)+\partial I_{1}(z) + \mut(z)I_{0}(z)= 2 \pi \mus(z) p_{0}(z) I_{0}(z) +q(z),
\]
and
\begin{equation}\label{infsys0}
\overline{\partial} I_{m}(z)+\partial I_{m+2}(z) + \mut(z)I_{m+1}(z)= 2 \pi \mus(z) p_{m+1}(z) I_{m+1}(z),\quad m\neq -1.
\end{equation}

In particular, in the non-attenuated and non-scattering case (when $\mut\equiv0\equiv\mus$)
the system corresponding to \eqref{infsys0} is 
\begin{equation}\label{infsys00}
\overline{\partial} J_{m}(z)+\partial J_{m+2}(z)= 0,\quad m\neq -1.
\end{equation}
The system \eqref{infsys00} was originally introduced (in a 
context more general than needed here) in \cite{bukhgeimBook},
and shown that their solutions satisfy a Cauchy like integral formula,
where the interior values in $D$ are recovered from their boundary values.
More precisely, for each $m\geq 0$, and each $z\in D$ fixed, 
\begin{equation}
J_m(z)
= \dfrac{1}{2\pi i}\int_{\partial\D}\dfrac{J_{m}(\zeta)}{\zeta-z}\:d\zeta 
+ \dfrac{1}{2\pi i}\int_{\partial\D} \left\{\dfrac{d\zeta}{\zeta-z} - \dfrac{d\overline{\zeta}}{\bar{\zeta}-\bar{z}}\right\}
  \sum_{j=1}^\infty J_{m+2j}(\zeta)\Biggl(\dfrac{\:\:\overline{\zeta}-\overline{z}\:\:}{\zeta-z}\Biggr)^j. \label{eq:BukhgeimCauchy}
\end{equation}

In the absorbing and non-scattering case ($\mus=0$ but $\mua\geq 0$)
the system \eqref{infsys0} becomes
\begin{equation}\label{infsys10}
\overline{\partial} I_{m}(z)+\partial I_{m+2}(z) + \mut(z)I_{m+1}(z)=0,\quad m\neq -1.
\end{equation}
The system \eqref{infsys10} has been studied first in \cite{ABK}
and shown that an integrating operator can be found to reduce it to \eqref{infsys00}.
We briefly describe the explicit construction of the integrating factor introduced in \cite{finch}
and its convolution form in \cite{sadiqTamasan02}
to be used in our reconstruction algorithms below.
For brevity assume that $\mua$ and $\mus$ are extended by zero outside the domain. 

For $(x,\xi) \in \Real^2\times S^1$,
let $\Div[\mut](x,\xi) = \int_0^\infty \mut(x+t\xi)\:dt$ be the divergent beam transform,
$\Radn[\mut](s,\xi) = \int_{-\infty}^{\infty} \mut(s\xi+t\xip)\:dt$ be the Radon transform,
and
$\Hilb[f](s) = \dfrac{1}{\pi}\pv\int_{-\infty}^\infty \dfrac{f(t)}{s-t}\:dt$ be the Hilbert transform,
where  $\xip \in S^1$ denotes the counterclockwise rotation of $\xi$ by $\pi/2$. Let us define
\begin{equation}\label{h}
h[\mut](x,\xi)
= \Div[\mut](x,\xi) - \dfrac{1}{2}(I-i\Hilb)\Radn[\mut](x\cdot\xip, \xip),
\end{equation}
where
$\displaystyle\Hilb\!\!\Radn[\mut](x\cdot\xip,\xip)= \Hilb\Bigl[\Radn[\mut](\cdot,\xip)\Bigr](x\cdot\xip)
= \dfrac{1}{\pi}\pv\int_{-\infty}^{\infty} \dfrac{\Radn[\mut](t,\xip)}{x\cdot\xip-t}dt$.
While there are many possible such integrating factors
to reduce the system \eqref{infsys10} to \eqref{infsys00},
the key feature of the construction in \cite{finch} is the fact
that all the negative modes of $\theta\mapsto h[\mut](x,\xi(\theta))$
(and thus of $e^{\pm h}$) vanish, as shown in \cite{finch,nattererBook, bomanStromberg}.
Let $\set{\alpha_k(x) \sep k \in \Z_{\ge 0}}$ and $\set{\beta_k(x) \sep k \in \Z_{\ge 0}}$
be the corresponding sequences of the Fourier modes, 
\begin{equation}\label{alphasbetas}
e^{-h[\mut](x,\xi)} = \sum_{k=0}^\infty \alpha_k(x) e^{ik\theta},
\quad
e^{h[\mut](x,\xi)} = \sum_{k=0}^\infty \beta_k(x) e^{ik\theta}.
\end{equation}
Then, as shown in \cite[Lemma 4.1]{sadiqTamasan02}, 
if
\begin{equation}\label{JfromI}
J_m = \sum_{k=0}^{\infty}\alpha_{k} I_{m+k},\quad m\geq 0,
\end{equation}
where $\set{I_m}$ is a solution to \eqref{infsys10}, then
$\set{J_m}$ is a solution of \eqref{infsys00}, and conversely, if 
\begin{equation}\label{IfromJ}
I_m = \sum_{k=0}^{\infty}\beta_{k} J_{m+k},
\end{equation}
where $\set{J_m}$ solves \eqref{infsys00}, 
then $\set{I_m}$ is a solution of \eqref{infsys10}.
Moreover, the Cauchy problem for \eqref{infsys00} (and thus for \eqref{infsys10}) has at most one solution.

\section{Reconstruction in the presence of scattering}\label{sec:procedure}
We consider now the scattering case, when the scattering kernel $p$ is of polynomial 
type in the angular variable, i.e.,
\[
p(z, \cos\theta)=\sum_{k=-M}^{M} p_k(z)e^{ik\theta}=p_0(z)+2\sum_{k=1}^M p_k(z)\cos(k\theta),
\]
where $M$ is the degree of the polynomial.
We stress that no smallness is assumed on the Fourier modes $p_k$, for $k=0,...,M$.

In this case the transport equation \eqref{TransportScatEq1} reduces to the system
\begin{subequations}
\begin{alignat}{2}
\overline{\partial} I^{(M)}_{-1} + \partial I^{(M)}_1 + \mut I^{(M)}_0 &= 2\pi \mus p_0 I^{(M)}_0 + q, & & \label{eq:reconst} \\
\overline{\partial} I^{(M)}_{m} + \partial I^{(M)}_{m+2} + \mut I^{(M)}_{m+1} &= 2\pi \mus p_{m+1} I^{(M)}_{m+1}, &\qquad& 0 \leq m \leq M-1, \label{eq:prepoisson}\\
\overline{\partial} I^{(M)}_{m} + \partial I^{(M)}_{m+2} + \mut I^{(M)}_{m+1} &= 0, &\qquad& m \ge M. \label{eq:prel2analytic}
\end{alignat}
\end{subequations}

The basic idea in the reconstruction starts from the observation
that the system \eqref{eq:prel2analytic} is of the type \eqref{infsys10}.
Therefore, via the integrating formulas \eqref{JfromI} for $m\geq M$,
the problem reduces to finding $\set{J_m^{(M)}}_{m\geq M}$ solution
of the Cauchy problem for the elliptic system \eqref{infsys00}.
Moreover, from the data on the boundary we can recover for each $m\geq M$,
\[
I^{(M)}_{m}(\zeta) = \dfrac{1}{2\pi} \int_0^{2\pi} I_{\text{measure}}\bigl(\zeta,\xi(\theta)\bigr) e^{im\theta}\:d\theta, \quad\zeta\in\Gamma, 
\]
and by using \eqref{JfromI} we find the boundary data 
\[
J^{(M)}_m\bigm|_{\Gamma} = \sum_{k=0}^\infty \alpha_k I_{k+m}\bigm|_\Gamma, \quad m\geq M.
\]

Next we use the Cauchy-like integral formula \eqref{eq:BukhgeimCauchy} to recover the interior values 
$\set{J_m(z) \sep z\in D}$ for $m\geq M$. 
Finally, we use \eqref{IfromJ} to recover the interior values $\set{I_m(z) \sep z\in D}$,
\[
I^{(M)}_m\bigm|_{\Gamma} = \sum_{k=0}^\infty \beta_k J_{k+m}\bigm|_\Gamma, \quad m\geq M.
\]

Recursively and in the decreasing order starting with the index $m=M-1$ to $m=0$,
we solve the elliptic problems \eqref{eq:prepoisson} as follows.
By applying $\partial$ to \eqref{eq:prepoisson}, we are lead to solving
\begin{subequations}\label{eq:poisson}
\begin{alignat}{2}
\triangle I^{(M)}_{m} &= 4\partial\left\{-\partial I^{(M)}_{m+2} + (2\pi \mus p_{m+1}-\mut) I^{(M)}_{m+1}\right\}, &\quad& \text{in $\D$}, \label{eq:poisson:eq}
\intertext{while on the boundary}
I^{(M)}_{m}(\zeta) &= \dfrac{1}{2\pi} \int_0^{2\pi} I_{\text{measure}}\bigl(\zeta,\xi(\theta)\bigr) e^{im\theta}\:d\theta, &\quad& \zeta\in\partial\D. \label{eq:poisson:bd}
\end{alignat}
\end{subequations}
for $0 \leq m \leq M-1$.
This is the boundary value problem of the Poisson equation in $D$.
Since $I_M^{(M)}\in H^1(D)$ the right hand side of \eqref{eq:poisson:eq} lies in $H^{-1}(D)$.
Since the trace at the boundary is in $H^{1/2}(\Gamma)$ the unique solution $I^{(M)}_{M-1} \in H^1(D)$,
and thus the regularity requirement needed to carry the argument to the next index down is satisfied.
Recursively, we recovered $I^{(M)}_{M-1}$, $I^{(M)}_{M-2}$, $\dotsc$, and $I^{(M)}_{0}$ in $D$.

Finally, since $I^{(M)}_{-1} = \overline{I^{(M)}_1}$, from the recovered $I_1^{(M)}$ and $I_0^{(M)}$,
we can now use \eqref{eq:reconst} to reconstruct the unknown source $q$ in $D$.

\begin{remark}
In general, the scattering kernel $p$ is not of polynomial type.
An immediate application of the well posedness of the forward problem in $L^2(D\times S^1)$
gives an error estimate in the measured outflow due to the approximation in the scattering kernel. 
\end{remark}

\begin{proposition}
Let $\mut,\mus, p$ be such that the forward problem \eqref{eq:forward} has
a unique solution and let $\tilde{p}$ be such that
\[
\int_{D\times S^1}\lvert p(x,\xi\cdot\xi')-\tilde{p}(x,\xi\cdot\xi')\rvert^2\:d\sigma_\xi\:dx\leq\epsilon^2,
\]
for some $0<\epsilon<1$, and let $\tilde{I}$ be the unique solution for \eqref{eq:forward} corresponding to $\tilde{p}$. Then,
\begin{equation}\label{dataNoise}
\norm{I|_{\Gamma_{+}}-\tilde{I}|_{\Gamma_{+}}} \leq C\epsilon,
\end{equation}
where $C>0$ is a constant depending only on $\mut,\mus, p$ and the domain.
\end{proposition}
\begin{proof}
It is easy to see that the difference of the two corresponding solutions $h=I-\tilde{I}$ satisfy
\begin{align*}
\xi\cdot\nabla_x h +\mut h &= \mus \int_{S^1} \tilde{p}h \;d\sigma_{\xi'} +\mus\int_{S^1}(\tilde{p}-p)I\; d\sigma_{\xi'},\\
h \bigm|_{\Gamma_{-}} &= 0,
\end{align*}
By interpreting the last term as a source and using the classical estimates in the forward model; see, e.g., \cite[Theorem 2.1]{stefanovUhlmann08}, we obtain
\begin{align*}
\norm{I|_{\Gamma_{+}}-\tilde{I}|_{\Gamma_{+}}}_{L^2(\Gamma_{+})} &\leq\sup_D \mus \norm{\int_{D\times S^1}\vert(\tilde{p}-p)I  \rvert \; d\sigma_{\xi'} dx}\\
&\leq\sup_D\mus \norm{I}_{L^2(D\times S^1)} \norm{p-\tilde{p}}_{L^2(D\times S^1)}\leq C\epsilon.
\end{align*}	
	
\end{proof}

For the two dimensional Henyey-Greenstein kernel \eqref{HG}, with $0\leq g<1$, considered in the numerical simulations
\[
p(\xi\cdot\xi')
= \dfrac{1}{2\pi} \dfrac{1-g^2}{1-2g \xi\cdot\xi' + g^2}
= \dfrac{1}{2\pi}\left(\sum_{m\in\Z} g^{|m|} e^{im\theta}\right)
\]
and its $M$-th order truncation
\begin{equation}\label{fourierExpHG}
p^{(M)}(\xi \cdot\xi')
= \dfrac{1}{2\pi}\left(\sum_{m\leq M} g^{|m|} e^{im\theta}\right),
\end{equation}
one obtains a refined estimate in terms of the anisotropic parameter $g$. Namely,
\begin{align*}
\int_{D\times S^1}\lvert p(\xi\cdot\xi')-p^{(M)}(\xi\cdot\xi')\rvert^2d\sigma_\xi dx
&=\pi\int_{S^1}\lvert p(\xi\cdot\xi')-p^{(M)}(\xi\cdot\xi')\rvert^2d\sigma_\xi
\\&=\sum_{m=M+1}^\infty g^{2m}=\frac{g^{2M+2}}{1-g^2},
\end{align*}
where the second last equality uses Parseval's identity.

For a given level of noise, a sufficiently large choice of $M$ yields that the difference \eqref{dataNoise}
between the exact data and the hypothetical data
(which would be obtained had the scattering been of polynomial type)
falls under the noise level.
Therefore our reconstruction produces a source for which the corresponding boundary data
is indistinguishable from the exact data within the level of noise.
This is the most we can hope to reconstruct. It is worth noting that,
for the scattering kernels of polynomial type, the method above does produce
the exact solution in a stable manner, as shown by the authors in \cite[Corollary 6.1]{FujiwaraSadiqTamasan}.

\section{Numerical Implementations}

\subsection{Evaluation of the Hilbert Transform}

An accurate calculation of the Hilbert transform is one of the crucial steps in our algorithm.
This transform appears in computing the integrating factor in \eqref{h}
and its Fourier modes in the angular variable \eqref{alphasbetas}.
In order to perform a reliable numerical integration  we need to properly account for the presence of the singularity in the kernel. The next simple lemma is key to our numerical treatment of the Hilbert transform.
\begin{lemma}
Suppose that $f \in C^1_0(\Real)$ and $\supp f \subset [a,b]$.
Then
\[
\pv\int_{-\infty}^{\infty} \dfrac{f(t)}{s-t}\:dt
= \begin{cases}
-\displaystyle\int_a^b g(s,t)\:dt + f(s)\log\dfrac{s-a}{b-s}, \quad & s \in (a,b);\\
\displaystyle\int_a^b \dfrac{f(t)}{s-t}\:dt, \quad &s \notin[a,b];\\
\displaystyle\int_a^b g_a(t)\:dt, \quad & s = a;\\
\displaystyle\int_a^b g_b(t)\:dt, \quad & s = b
\end{cases}
\]
where $g$, $g_a$, $g_b$ are bounded and continuous functions defied by
\begin{align*}
g(s,t) &= \begin{cases}
\dfrac{f(s)-f(t)}{s-t}, \quad & s \neq t;\\
f'(s),                  \quad & s = t.
\end{cases}
\\
g_a(t) &= \begin{cases}
\dfrac{f(t)}{a-t}, \quad & t > a;\\
0, \quad & t = a,
\end{cases}
\\
g_b(t) &= \begin{cases}
\dfrac{f(t)}{b-t}, \quad & t < b;\\
0, \quad & t = b.
\end{cases}
\end{align*}
It means that the integrals on the right hand side is those in the sense of Riemann.
\end{lemma}
\begin{proof}
Firstly, for $s \in (a,b)$, then
\begin{align*}
\pv\int_{-\infty}^\infty \dfrac{f(t)}{s-t}\:dt
&= \lim_{\epsilon\downarrow0}\biggl\{
        - \int_a^{s-\epsilon} \dfrac{f(s)-f(t)}{s-t}\:dt
        + \int_a^{s-\epsilon} \dfrac{f(s)}{s-t}\:dt
\\
&\qquad\qquad
        - \int_{s+\epsilon}^b \dfrac{f(s)-f(t)}{s-t}\:dt
        + \int_{s+\epsilon}^b \dfrac{f(s)}{s-t}\:dt \biggr\}.
\intertext{since $f\in C^1$ the function $g(s,t)$ is bounded continuous function on $\Real^2$,
and thus the calculation is followed by}
&= -\int_a^b g(s,t)\:dt + f(s)\lim_{\epsilon\downarrow0}
        \left\{\int_a^{s-\epsilon} \dfrac{dt}{s-t}
        + \int_{s+\epsilon}^b \dfrac{dt}{s-t}\right\}
\\
&= f(s)\log\dfrac{s-a}{b-s} -\int_a^b g(s,t)\:dt.
\end{align*}

Secondly, for $s \notin [a,b]$, the integrand of the transform $\dfrac{f(s)}{s-t}$ is regular
as a function of $t\in[a,b]$.

Finally, for $s=a$, then
\begin{align*}
\pv\int_{-\infty}^\infty \dfrac{f(t)}{a-t}\:dt
&= \lim_{\epsilon\downarrow0}\left(\int_{-\infty}^{a-\epsilon}\dfrac{f(t)}{a-t}\:dt
                             +\int_{a+\epsilon}^{\infty}\dfrac{f(t)}{a-t}\:dt\right)
\\
&= \lim_{\epsilon\downarrow0}\int_{a+\epsilon}^{\infty}\dfrac{f(t)}{a-t}\:dt,
\end{align*}
since $f(t)\equiv 0$ for $t < a-\epsilon$.
By virtue of $f\in C^1$ and $f(a) = f'(a) = 0$, $g_a$ is bounded and continuous on the interval $ t \ge a$.
Similar consideration works for the case $s=b$, which completes the proof.
\end{proof}

Another choice of non-zero interval of $f$ may give a different expression.
For example, if we adopt $[a',b'] \supset [a,b]$, then
\[
\pv\int_{-\infty}^\infty \dfrac{f(t)}{s-t}\:dt
= f(s)\log\dfrac{s-a'}{b'-s} -\int_{a'}^{b'} g(s,t)\:dt.
\]
It is easily seen by calculation that both expressions are equivalent.
Therefore we can theoretically choose any non-zero interval to evaluate
the Hilbert transform as the Riemann integral of bounded and continuous functions.

For the case of $a < s < b$, we split the interval at $s$ in numerical computation:
\[
\pv\int_{-\infty}^\infty \dfrac{f(t)}{s-t}\:dt
= f(s)\log\dfrac{s-a}{b-s} - \int_a^s g(s,t)\:dt - \int_s^b g(s,t)\:dt.
\]
because $g(s,t)$ may not be smooth on $t=s$.
It is also convenient to employ the mid-point rule
in order to avoid implementation of $f'$ appeared in $g$.

\subsection{Computation of Cauchy-type Integral Formula}\label{sec:ComplexIntegral}
Assume that $\partial\D$ has a parameterization $\zeta(\omega)$, $0 \leq \omega < 2\pi$
and $D$ contains the origin for simplicity.
Suppose that $0 = \omega_0 < \omega_1 < \dotsb < \omega_{K-1} < \omega_K = 2\pi$.
Take $\zeta_k \in \partial\D$ with $\Arg\zeta_k \in [\omega_{k}, \omega_{k+1})$,
where $\Arg\zeta \in [0,2\pi)$ is the argument of $\zeta \in \C$.
Let us consider the discretization of the complex integral of
an integrable function
$f(\zeta) = \displaystyle\sum_{k=0}^{K-1} f_k \Psi_k(\zeta)$
on $\partial\D$,
which is
\[
\int_{\partial\D}f(\zeta)\:d\zeta
= \int_0^{2\pi} f\bigl(\zeta(\omega)\bigr)\zeta'(\omega)\:d\omega
= \sum_{k=0}^{K-1} f_k \int_0^{2\pi} \Psi_k\circ\zeta(\omega)\zeta'(\omega)\:d\omega.
\]
We consider two examples as $\Psi_k$.
If $\Psi_k\circ\zeta$ is the characteristic function of the interval $[\omega_k,\omega_{k+1}]$, 
then the integral is approximated by
\begin{equation}\label{eq:RiemannSum}
\int_{\partial\D}f(\zeta)\:d\zeta
= \sum_{k=0}^{K-1} f_k \int_{\omega_k}^{\omega_{k+1}} \zeta'(\omega)\:d\omega
\approx \sum_{k=0}^{K-1} f_k \zeta'_k(\omega_{k+1} - \omega_{k}),
\end{equation}
where $\zeta'_k$ is the derivative at a certain point in the interval $[\omega_k,\omega_{k+1})$.
On the other hand, 
if one can choose 
$\Psi_k\circ\zeta$ be the continuous and piecewise-linear function
with $\Psi_k\circ\zeta(\omega_{\ell}) = \delta_{k\ell}$ (Kronecker's delta) and $\omega_k = \Arg\zeta_k$,
then the trapezoidal rule gives an approximation as
\[
\int_{\partial\D}f(\zeta)\:d\zeta
= \sum_{k=0}^{K-1} f_k \int_{\omega_{k-1}}^{\omega_{k+1}} \Psi_k\bigl(\zeta(\omega)\bigr)\zeta'(\omega)\:d\omega
\approx \sum_{k=0}^{K-1} f_k \zeta'_k\dfrac{\Arg\zeta_{k+1} - \Arg\zeta_{k-1}}{2},
\]
where $\zeta'_k = \zeta'(\omega_k)$, $\zeta_{-1} = \zeta_{K-1} - 2\pi$ and $\zeta_K = \zeta_0 + 2\pi$ by virtue of the periodicity.


For the purpose of numerical computation,
we express
the integral in the second term on the right hand side of \eqref{eq:BukhgeimCauchy} as
\[
\dfrac{1}{2\pi i}\int_{\partial\D} \left(\dfrac{d\zeta}{\zeta-z} - \dfrac{d\overline{\zeta}}{\overline{\zeta}-\overline{z}}\right)F(\zeta)
= \dfrac{1}{\pi}\int_{0}^{2\pi} \Impart\left(\dfrac{\zeta'(\omega)}{\zeta(\omega)-z}\right)F\bigl(\zeta(\omega)\bigr)\:d\omega,
\]
with $F$ being the series in \eqref{eq:BukhgeimCauchy}.

\subsection{Proposed Algorithm}
In this subsection, we present the numerical algorithm for the source reconstruction.

Suppose that $\mua$, $\mus$, and $p$ on $\D$ are known.
Assume that $\partial\D$ has a smooth parameterization $\zeta(\omega)$,
$0 \leq \omega < 2\pi$.
Without loss of generality, we can assume that the domain $D$ contains the origin.
$I_{\text{measure}}(\zeta,\xi)$ on $\Gamma_{+}$ are
sampled at $\bigl(\zeta_k, \xi(\theta_n)\bigr)$, where
$\zeta_0$, $\dotsc$, $\zeta_{K-1}$ are $K$ distinct points on $\partial\D$,
and $\theta_n = 2\pi n/N$.
We assume that $0 \leq \Arg\zeta_0 < \Arg\zeta_1 < \dotsb < \Arg\zeta_{K-1} < 2\pi$.

\begin{step}
Fix positive integers $M$ and $S$.
The integer $S$ should be chosen sufficiently large, at least $S \ge M+3$.
Let 
$\zeta_k' = \zeta'(\Arg\zeta_k)$, $0 \leq k < K$.
Choose the segmentation $0=\omega_0 < \omega_1 < \dotsb < \omega_{K} =2\pi$ so that
$\Arg \zeta_k \in [\omega_k, \omega_{k+1})$.
From the periodicity we assert that $\omega_{j+K} = \omega_{j}+2\pi$, $j \in \Z$.
We introduce an inscribed polygonal domain $D_h \approx \D$, and 
take a triangulation $\mathcal{T} = \set{\tau_{\ell}}$ of $\D_h$,
i.e.\ each $\tau_{\ell}$ is a triangular domain, $\tau_{\ell}\cap\tau_k = \emptyset$ if $\ell\neq k$,
and $\overline{D_h} = \displaystyle\bigcup_{\ell}\overline{\tau_{\ell}}$.
Let $P_1(\mathcal{T})$ denote the set of the piecewise linear continuous functions
with respect to $\mathcal{T}$.
We denote by $V$ the set of vertices of $\mathcal{T}$.
\end{step}


\begin{step}\label{step:IntegratingFactor}
Compute
\begin{align*}
\alpha_{s,k} &= \dfrac{1}{N}\sum_{n=0}^{N-1} \exp\Bigl(-h[\mut](\zeta_k,\xi(\theta_n))\Bigr) e^{-is\theta_n},
\intertext{and}
\beta_{s,\ell} &= \dfrac{1}{N}\sum_{n=0}^{N-1} \exp\Bigl(h[\mut](z_\ell,\xi(\theta_n))\Bigr) e^{-is\theta_n},
\end{align*}
for $0 \leq s \leq S$, $0 \leq k < K$, and $z_{\ell} \in V$.
The function $h[\mut]$ is evaluated by the use of mid-point rule as stated so far.
\end{step}

\begin{step}\label{step:bd}
Compute
\[
\mathcal{I}_{m,k} = \dfrac{1}{N} \sum_{n=0}^{N-1} I_{\text{measure}}\bigl(\zeta_k,\xi(\theta_n)\bigr) e^{im\theta_n},
\]
for $0 \leq m \leq S$ and $0 \leq k < K$.
\end{step}

\begin{step}\label{step:convolution}
Compute
\[
\mathcal{J}_{m,k}
= \sum_{0 \leq s,s+m \leq S} \alpha_{s,k} \mathcal{I}_{s+m,k},
\]
for $M \leq m \leq S$ and $0 \leq k < K$.
\end{step}

\begin{step}\label{step:BoundaryIntegral}
Compute $\mathcal{J}_{m}(z_{\ell})$ for $M \leq m \leq S-2$ and $z_{\ell} \in V\cap\D$,
where
\begin{multline}\label{eq:BukhgeimCauchyNum}
\mathcal{J}_{m}(z) = \dfrac{1}{2\pi i} \sum_{k=0}^{K-1}\dfrac{\zeta'_k}{\zeta_k - z}\mathcal{J}_{m,k}\Delta\omega_k\\
 + \dfrac{1}{\pi} \sum_{k=0}^{K-1}\left\{\Impart \left(\dfrac{\zeta'_k}{\zeta_k - z}\right)\right\}
\left\{
  \sum_{m+2 \leq m+2j \leq S} \mathcal{J}_{m+2j,k}\Biggl( \dfrac{\:\overline{\zeta_k} - \overline{z}\:}{\zeta_k - z} \Biggr)^j
\right\}\Delta\omega_k,
\end{multline}
with $\Delta\omega_k = \omega_{k+1}-\omega_{k}$,
implied by \eqref{eq:RiemannSum}.
We can change it to
$\Delta\omega_k = (\Arg\zeta_{k+1}-\Arg\zeta_{k-1})/2$
if the piecewise-linear approximation to $J_m\bigm|_{\D}$ is valid.
\end{step}

\begin{step}\label{step:interpolationJboundary}
Compute $\mathcal{J}_{m}(\zeta_{\ell})$ for $M \leq m \leq S$ and $\zeta_{\ell} \in V\cap\partial\D$
by interpolating $\set{\mathcal{J}_{m,k} \sep 0 \leq k < K }$ obtained in \cref{step:convolution}.
\end{step}

\begin{step}\label{step:decomvolution}
For $z_{\ell} \in V$, compute
\begin{align*}
\mathcal{I}_{M}(z_{\ell}) &= \sum_{0 \leq s, s+M \leq S-2} \beta_{s,\ell} \mathcal{J}_{s+M}(z_{\ell}),
\intertext{and}
\mathcal{I}_{M+1}(z_{\ell}) &= \sum_{0 \leq s, s+M+1 \leq S-2} \beta_{s,\ell} \mathcal{J}_{s+M+1}(z_{\ell}).
\end{align*}
\end{step}

\begin{step}\label{step:boundary}
For $m=M+1, M, \dotsc, 1,0$ (in descending order),
find a piecewise linear continuous function
$\mathcal{I}_{m}|_{\partial\D} = \sum a_{\ell} \varphi_{\ell}$
by interpolating $\set{\mathcal{I}_{m,k} \sep 0 \leq k < K }$ obtained in \cref{step:bd},
where $\varphi_{\ell}$ is a periodic and piecewise linear continuous function on $\partial\D$
with $\varphi_{\ell}(v_k) = \delta_{{\ell}k}$ (Kronecker's delta), $v_k\in V\cap\partial\D$.
A more detailed example follows the algorithm description.
\end{step}

\begin{step}
For $m=M-1, M-2, \dotsc, 1,0$ (in descending order),
find an approximation $\mathcal{I}_{m} \in P_1(\mathcal{T})$ to $I_m \in H^1(D)$
by solving the Dirichlet problem of the Poisson equation \eqref{eq:poisson}
with the standard $P_1$ finite element method~\cite{ErnGuermond}.
The variational formulation for \eqref{eq:poisson} is
approximated as follows;
Find $\mathcal{I}_{m} \in P_1(\mathcal{T})$ with \eqref{eq:poisson:bd}
so as to satisfy
\begin{align*}
-\int_{D_h} \nabla{\mathcal{I}_{m}}\cdot\nabla\varphi\:dx
&= 4\int_{D_h}\partial\left\{-\partial \mathcal{I}_{m+2} + (2\pi \mus p_{m}-\mut) \mathcal{I}_{m+1}\right\}\varphi\:dx
\\
&= -4\sum_{\tau\in\mathcal{T}}\int_{\tau}\left\{-\partial \mathcal{I}_{m+2} + (2\pi \mus p_{m}-\mut) \mathcal{I}_{m+1}\right\}\partial\varphi\:dx,
\end{align*}
for any $\varphi\in P_1(\mathcal{T})$ with $\varphi|_{V\cap\partial\D_h} = 0$.

Firstly for $m=M-1$,
\cref{step:decomvolution} and \cref{step:boundary} give
$\mathcal{I}_{M+1}$ and $\mathcal{I}_{M}$ on the right hand side at $z_{\ell} \in V$,
which leads the interpretation $\mathcal{I}_{M+1}, \mathcal{I}_{M} \in P_1(\mathcal{T})$.
Particularly, if $\mathcal{I}_{M+1}(x)= a_{\ell}x_1 + b_{\ell}x_2 + c_{\ell}$ on a triangle $\tau_{\ell}$,
then $\partial \mathcal{I}_{M+1}|_{\tau_{\ell}} = (a_{\ell} - ib_{\ell})/2$.
Similarly for the test function $\varphi\in\mathcal{P}_1(\mathcal{T})$, we can find $\partial\varphi$.
The integration on the right hand side can be evaluated by
a Gauss-type numerical integration~\cite{ErnGuermond} on each triangle $\tau$.
Then we can obtain $\mathcal{I}_{m-1} \in P_1(\mathcal{T})$
by solving the linear system.

For $m=M-2,M-3,\dotsc,1,0$, we can find
approximations $\mathcal{I}_{m} \in P_1(\mathcal{T})$ similarly.
\end{step}

\begin{step}
For each triangle $\tau_{\ell} \in \mathcal{T}$,
let $\mathcal{I}_1|_{\tau_{\ell}} = a_{\ell} x_1 + b_{\ell} x_2 + c_{\ell}$.
Then reconstruction of $q|_{\tau_{\ell}}$ is given by \eqref{eq:reconst} as
\[
q_{\ell}
=
\Repart(a_{\ell}) + \Impart(b_{\ell})
 + \bigl\{\mut(z_{\ell})-2\pi\mus(z_{\ell})p_0(z_{\ell})\bigr\}\Repart\bigl(\mathcal{I}_0(z_{\ell})\bigr).
\]
This ends the Algorithm.
\end{step}

\bigskip

In \cref{step:boundary},
there may be a mismatch between the measurement points on the boundary
and vertices of the triangulation for reconstruction.
We solve this mismatch by interpolating the data on formers.
Below we detail an example.

Let us assume that $\sum_k \mathcal{I}_{m,k}\chi_k$
gives an interpolation of $\set{\mathcal{I}_{m,k} \sep 0 \leq k < K }$,
where $\chi_{k}$ is square integrable on $\partial\D$.
Then $a_1,\dotsc,a_L$ can be determined as its best approximation
in the sense of least square
\[
\min_{a_1,\dotsc,a_L}\norm{\sum_{\ell} a_{\ell}\varphi_{\ell} - \sum_k \mathcal{I}_{m,k}\chi_k}_{L^2(\partial\D)}.
\]
It is clear that there exists a unique minimizer.
In particular, if the $\D$ is the unit circle and the measurement points $\zeta_k$ are equi-spaced,
then minimizer satisfies the system of linear equations
\[
\dfrac{2\pi}{L}
\begin{pmatrix}
\tfrac{2}{3} & \tfrac{1}{6} & 0 & \cdots & \tfrac{1}{6} \\
\tfrac{1}{6} & \tfrac{2}{3} & \tfrac{1}{6} & \cdots & 0 \\
& \ddots & \ddots & \ddots & \\
0   & \cdots & \tfrac{1}{6} & \tfrac{2}{3} & \tfrac{1}{6} \\
\tfrac{1}{6} & 0 & \cdots & \tfrac{1}{6} & \tfrac{2}{3} \\
\end{pmatrix}
\begin{pmatrix}
a_1 \\ a_2 \\ \vdots \\ a_{L-1} \\ a_L
\end{pmatrix}
=
\begin{pmatrix}
b_1 \\ b_2 \\ \vdots \\ b_{L-1} \\ b_L
\end{pmatrix},
\]
where 
\[
b_{\ell} =\sum_{k} \mathcal{I}_{m,k} \int_{0}^{2\pi} \chi_k(\theta)\varphi_{\ell}(\theta)\:d\theta.
\]
Hence the boundary value is given by $\mathcal{I}_{m}(v_{\ell}) = a_{\ell}$.
The similar strategy is applicable for \cref{step:interpolationJboundary}.


\bigskip

\subsection{A locally optimal truncation criterion}
We give a criteria on the choice of $M$.
In the proposed algorithm,
$\mathcal{I}_0$ is obtained in complex values.
On the other hand, the exact value of the zero-th Fourier mode $I_0$ is real-valued since $I$ is so.
Therefore the imaginary part of $\mathcal{I}_0$ comes as errors in reconstruction.
It is reasonable to consider that the errors in the real and the imaginary part interact each other.
This observation leads a choice of $M$ so as to minimize the imaginary part of $I_0$,
which is expected to reduce the error in the real part efficiently.
Based on this consideration,
we call $M$ optimal which attains a \emph{local} minimum of the imaginary part
\begin{equation}\label{eq:error:imag}
E_{\text{imag}}
= \left\{\sum_{\tau \in \mathcal{T}}
  \bigl|\tau_{\ell}\bigr|
  \bigl\{\mut(z_{\ell})-2\pi\mus(z_{\ell})p_0(z_{\ell})\bigr\}^2
  \bigl(\Impart\mathcal{I}_0(z_{\ell})\bigr)^2
\right\}^{1/2},
\end{equation}
where
$|\tau_{\ell}|$ is the area of $\tau_{\ell}$ and
$\bigl\{\mut(z_{\ell})-2\pi\mus(z_{\ell})p_0(z_{\ell})\bigr\}\Impart\mathcal{I}_0(z_{\ell})$
corresponds to \emph{the imaginary part of reconstructed $q$} on $\tau_{\ell}$.
In order to obtain an optimal $M$, we reconstruct the source for several values of $M$,
then choose a value which minimizes $E_{\text{imag}}$.
We stress here  that the optimality indicator in \eqref{eq:error:imag}
does not require knowledge of the unknown source.

\section{Numerical Experiments}\label{sec:numstudy}

In this section we demonstrate the numerical feasibility of the proposed algorithm
for two numerical examples.
All computations are processed with IEEE754 double precision arithmetic.
In both numerical experiments, the measurement data is generated by
solving the forward problem ~\eqref{eq:forward}
by the piecewise-constant upwind approximation~\cite{fujiwaraFVM}.
Therefore it is natural to use the characteristic function as the interpolation basis
$\Psi_k$ in \Cref{sec:ComplexIntegral}, and
$\chi_k$ in \cref{step:boundary}.
The triangulation is generated by FreeFem++~\cite{FreeFemPP}.

The scattering coefficient is $\mus \equiv 5$. Physically,
It means that the particle scatters on average every $1/5$ unit of length.
Given that $D$ is the unit disc, particles scatter on average $10$ times before getting out.
We use the the two dimensional version of the Henyey-Greenstein scattering kernel in \eqref{HG}
\[
p(\xi\cdot\xi')
= \dfrac{1}{2\pi} \dfrac{1-g^2}{1-2g \xi\cdot\xi' + g^2}
\]
with the anisotropy parameter $g=1/2$. This choice if half way between the ballistic $g=1$
and isotropic $g=0$ case. The Fourier expansion of the scattering kernel in \eqref{fourierExpHG}
\[
\dfrac{1}{2\pi}\dfrac{1-g^2}{1-2g \cos\theta + g^2}
= \dfrac{1}{2\pi}\left(\sum_{m\in\Z} g^{|m|} e^{im\theta}\right),
\]
yields the simple form of the modes $p_m = g^m/2\pi$, for all $m\geq 0$.

In the discretization of the boundary, we adopt $S=128$,
while for the mid-point rule in the computations of the integral transforms, we use $100$ sampling points.


\begin{experiment}[\cite{FujiwaraSadiqTamasan}]
Let
\begin{align*}
R
&= (-0.25,0.5)\times(-0.15,0.15),
\\
B_1
&= \set{ (x_1,x_2) \sep (x_1-0.5)^2 + x_2^2 < 0.3^2 },
\\
B_2
&= \left\{ (x_1,x_2) \sep \left(x_1+0.25\right)^2 + \left(x_2-\dfrac{\sqrt{3}}{4}\right)^2 < 0.2^2 \right\}.
\end{align*}
be the rectangular, respectively circular subsets of $\D$ as shown in \cref{fig:inclusions}.

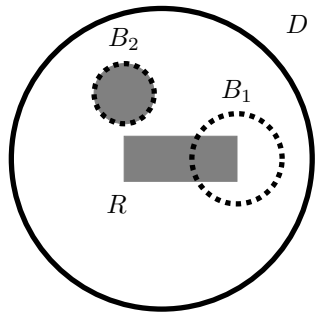
\begin{figure}[ht]
\centering
\begin{tikzpicture}[scale=2]
\draw [line width=2] (0,0) circle (1); 
  \node at (0.9,0.9) {$\D$};

\draw [color=gray,fill=gray,line width=0] (-0.25,0.433013) circle (0.2); 

\draw [color=gray,fill=gray,line width=0] (-0.25,-0.15) rectangle (0.5,0.15); 
  \node at (-0.3,-0.3) {$R$};

\draw [dotted,line width=2] (0.5,0) circle (0.3); 
  \node at (0.5,0.45) {$B_1$};

\draw [dotted,line width=2] (-0.25,0.433013) circle (0.2); 
    \node at (-0.25,0.8) {$B_2$};
\end{tikzpicture}
\caption{\label{fig:inclusions}
Locations of inclusions in numerical examples.
The medium is relatively strongly absorbing inside the dotted balls,
while the source $q(x)$ is located in the gray regions.
}
\end{figure}

The source, to be reconstructed, is 
\begin{equation}\label{eq:source}
q(x) = \begin{cases}
2,  \qquad & \text{in $R$};\\
1,  \qquad & \text{in $B_2$};\\
0,  \qquad & \text{otherwise}.
\end{cases}
\end{equation}
The absorption coefficient $\mua(x)$ is given by 
\[
\mua(x) = \begin{cases}
2, \quad &\text{in $B_1$};\\
1, \quad &\text{in $B_2$};\\
0.1, \quad &\text{otherwise}.
\end{cases}
\]

To generate the boundary data, we solve the forward problem using the numerical method in \cite{fujiwaraFVM}
with a triangular mesh of $5,542,718$ triangles,
and $360$ equispaced velocity intervals to describe the velocity directions.
We disregard the value $I(x,\xi)$ of the solution for $x\in D$, and only keep the boundary values.
The obtained boundary data $I(x,\xi)$ on $\partial\D\times S^1$
is depicted in \cref{fig:example1:boundary}.
In this figure, for $x \in \partial\D$ indicated by cross symbols $(\times)$,
the graph of $\bigl(I(\zeta,\xi), \xi\bigr)$, $\xi \in S^1$ are
shown by a red (closed) curve in the polar coordinate with the center at each $\zeta$.
In other words, the red curve is the graph of
$\set{ \zeta + 2I(\zeta,\xi)\xi \sep \xi \in S^1}$
for $\zeta \in \partial\D$ indicated by the cross symbols.
By the assumption of no incoming radiation, i.e.\ $I|_{\Gamma_{-}} = 0$,
the red curve never appear inside $|x| < 1$ indicated by gray in this expression.
\begin{figure}[ht]
\begin{minipage}{.6\textwidth}
\centering
\includegraphics[width=.8\textwidth]{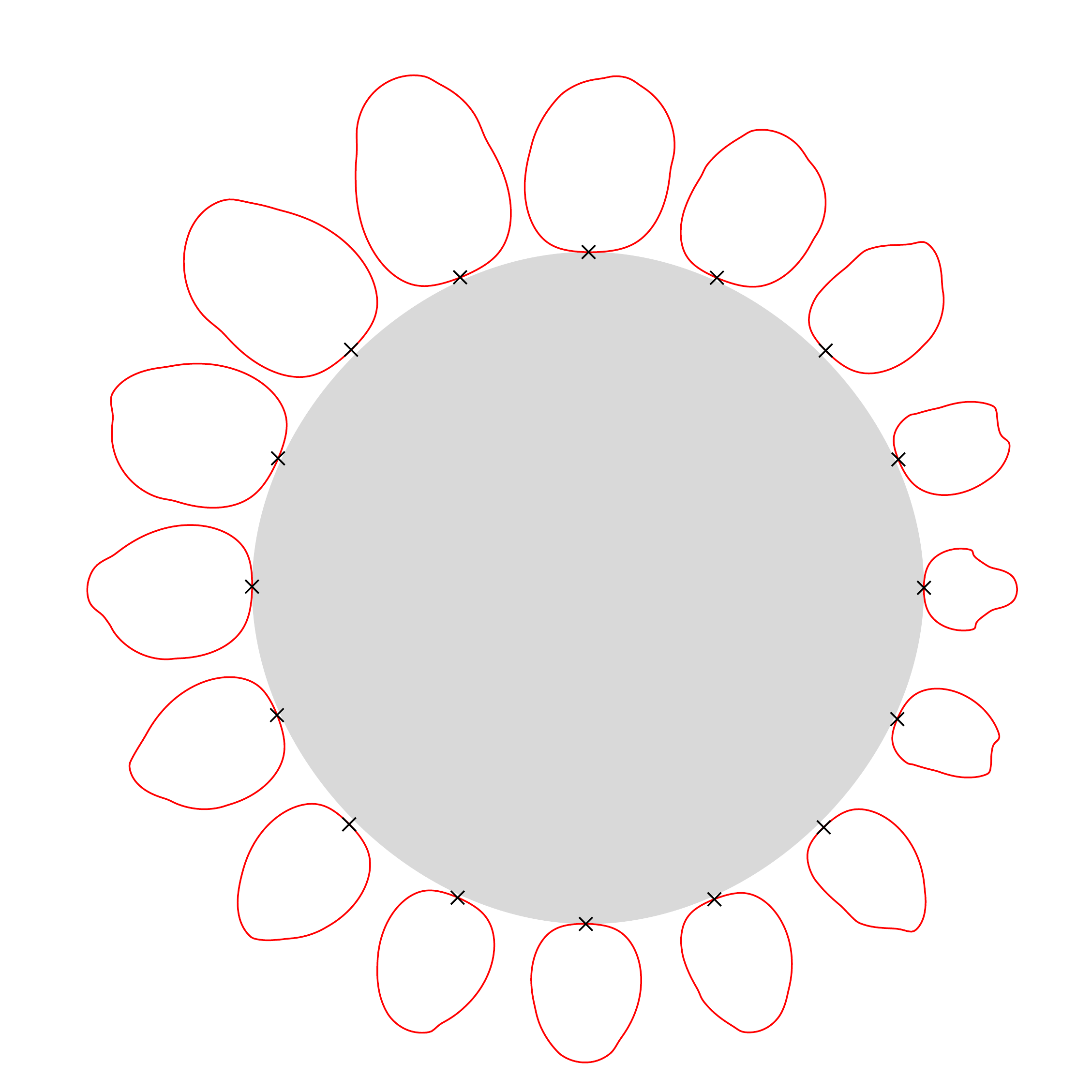}
\end{minipage}
\hfill
\begin{minipage}{.38\textwidth}
\centering
\includegraphics[width=\textwidth]{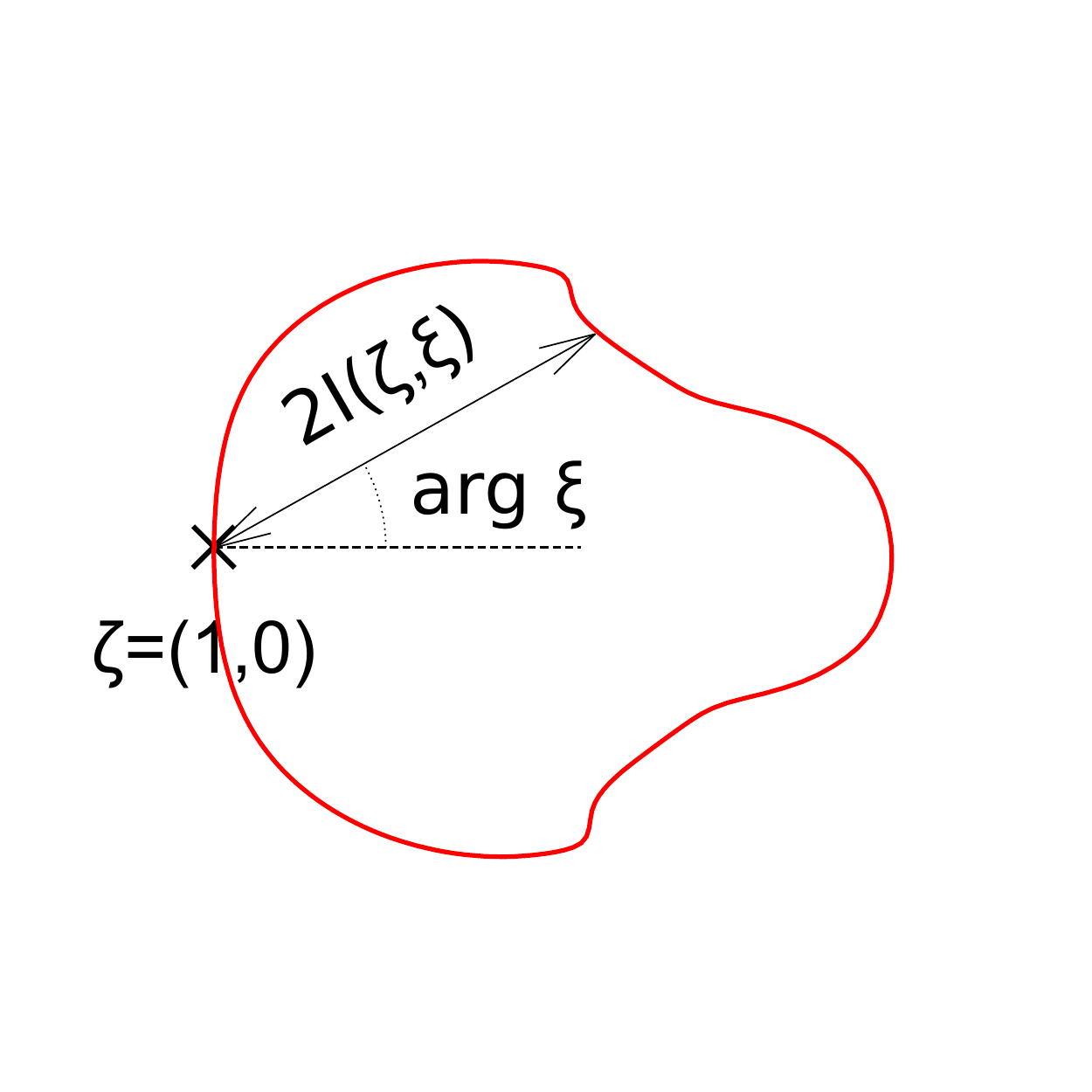}
\end{minipage}
\caption{\label{fig:example1:boundary}Boundary measurement $I(\zeta,\xi)|_{\partial\D\times S^1}$
obtained by the numerical computation of the forward problem in the unit disc (in grey). For $\zeta\in\partial\D$
(indicated by $\times$), the red curve is $\set{\zeta + 2 I(\zeta,\xi)\xi \sep \xi \in S^1}$ (on the left).
The right figure is a magnification of the curve at $\zeta=(1,0)$.
}
\end{figure}
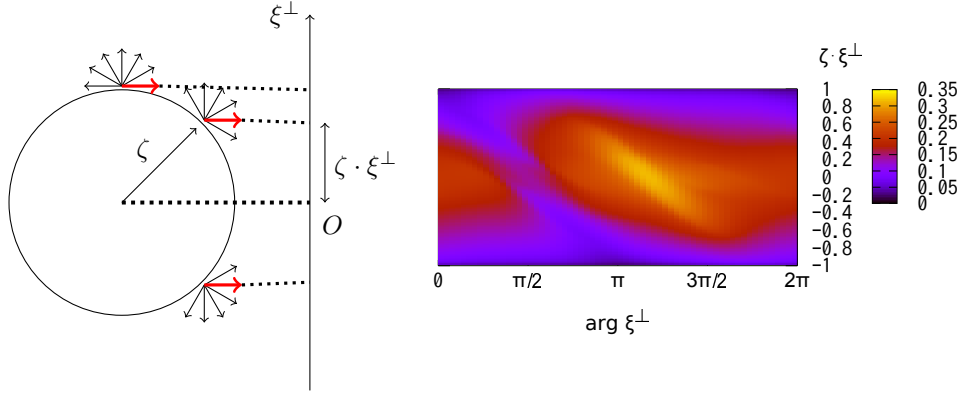
\begin{figure}[ht]
\begin{minipage}{.41\textwidth}
\begin{tikzpicture}

\draw (0,0) circle (1.5);
\draw [->] (2.5,-2.5) -- (2.5,2.5);
  \node [left] at (2.5,2.5) {$\xip$};
\draw [dotted,line width=1.5pt] (0,0) -- (2.5,0);
  \node at (2.8,-0.3) {$O$};

\draw [->] (0,0) -- +(45:1.4);
  \node [left] at (0.5,0.7) {$\zeta$};

\draw [<->] (2.7,0) -- (2.7,1.06);
\node [right] at (2.7,0.5) {$\zeta\cdot\xip$};

  \foreach \t in {30,60,...,180}
    \draw [->] (90:1.55) -- +(\t:0.5);
  \draw [->,color=red,line width=1.2pt] (90:1.55) -- +(0:0.5);
  \draw [dotted,line width=1.2pt] (90:1.55)+(0.5,0) -- (2.5,1.5);

  \foreach \t in {-30,30,60,90,120}
    \draw [->] (45:1.55) -- +(\t:0.5);
  \draw [->,color=red,line width=1.2pt] (45:1.55) -- +(0:0.5);
  \draw [dotted,line width=1.2pt] (45:1.55)+(0.5,0) -- (2.5,1.06);

  \foreach \t in {-120,-90,-60,-30,30}
    \draw [->] (-45:1.55) -- +(\t:0.5);
  \draw [->,color=red,line width=1.2pt] (-45:1.55) -- +(0:0.5);
  \draw [dotted,line width=1.2pt] (-45:1.55)+(0.5,0) -- (2.5,-1.06);

\end{tikzpicture}
\end{minipage}
\begin{minipage}{.58\textwidth}
\includegraphics[width=\textwidth]{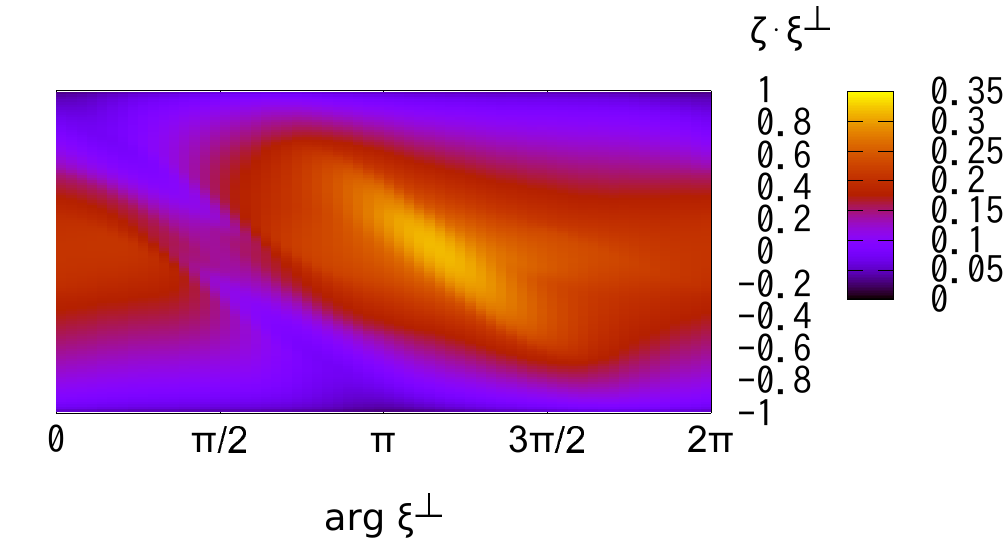}
\end{minipage}
\caption{\label{fig:example1:sinogram}
Left : Projection of the measurement data $I(\zeta,\xi)|_{\Gamma_{+}}$.
The red arrows, $I(\zeta,\xi)$ with $\xi=(1,0)$, are projected to the plane with $\Arg\xip=\pi/2$.
Right : Projection of $I|_{\Gamma_{+}}$ corresponding to \cref{fig:example1:boundary},
the horizontal axis is the argument of the projection plane,
and the vertical axis is the distance from the projected origin.
}
\end{figure}


To better exhibit the effect of scattering, in the second representation of the data $I\big|_{\Gamma_{+}}$
in \cref{fig:example1:sinogram} we use the same coordinates as those used in a classical sinogram
for the Radon transform data.
More precisely, if $\xip \in S^1$ denote the counterclockwise rotation of $\xi$ by $\pi/2$,
then the horizontal axis is the argument of the projection plane with direction $\xip$,
while the vertical axis is the distance from the projected origin.
In the absence of scattering, $p=0$, this representation would be exactly the sinogram
for the attenuated Radon transform data.
Note in \cref{fig:example1:sinogram} how the scattering had combined
and smeared out the features from the three locations of the source.

To avoid an inverse crime, the triangulation used in the reconstruction is different from that in the forward problem.
In particular the reconstruction mesh consists of $6,998$ triangles with $5,400$ vertices
(much less than the $5,542,718$ triangles used in the forward problem),
and is generated without any information of the location of the subsets $R$, $B_1$, and $B_2$.
The computational time for reconstruction with $M=6$ is $340$ seconds
on Xeon E5-2650 v4 (2.2GHz, 12 cores) with OpenMP.
Almost all computational time are occupied by the
computation of the discrete Fourier transform of $e^{-h}$ and $e^{h}$ in \cref{step:IntegratingFactor},
and the boundary integral \eqref{eq:BukhgeimCauchyNum} in \cref{step:BoundaryIntegral},
as $117$ seconds, $79$ seconds, and $141$ seconds respectively.

The reconstructed $q(x)$ with $M=6$ is shown in \cref{fig:disconta:reconst}.
Its cross sections along the dotted diameters $x_2=-\sqrt{3}x_1$ and $x_2 = 0$,
passing through the origin and the center of $B_2$ and $R$ respectively,
are depicted in \cref{fig:disconta:reconst:section}. 
The reconstructed $q(x)$ shows a quantitative agreement with the exact source in \eqref{eq:source}.
Similar to the X-ray and attenuated X-ray tomography, the artifacts appear due to the co-normal singularities
in the source but also in the attenuation. 

\begin{figure}[ht]
\centering
\begin{minipage}{.52\textwidth}
\includegraphics[width=\textwidth]{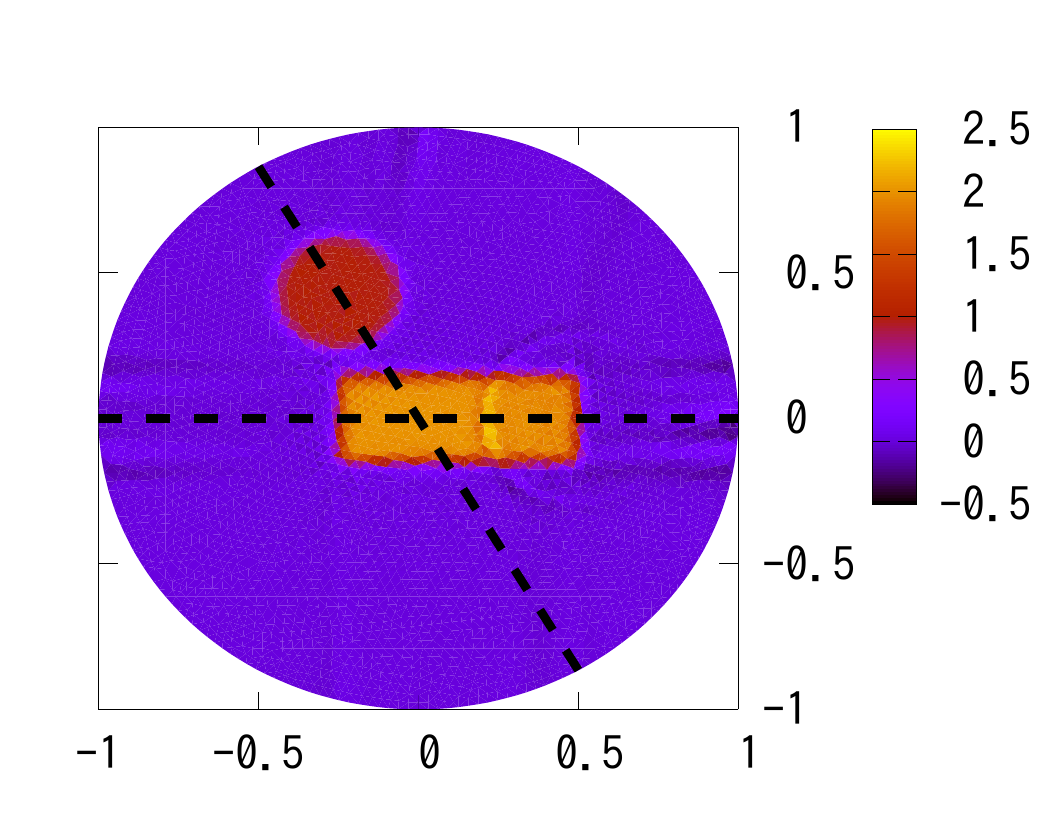}
\end{minipage}
\caption{\label{fig:disconta:reconst} Reconstructed source with $M=6$}
\end{figure}
\begin{figure}
\begin{minipage}{.43\textwidth}
\centering
\includegraphics[width=\textwidth]{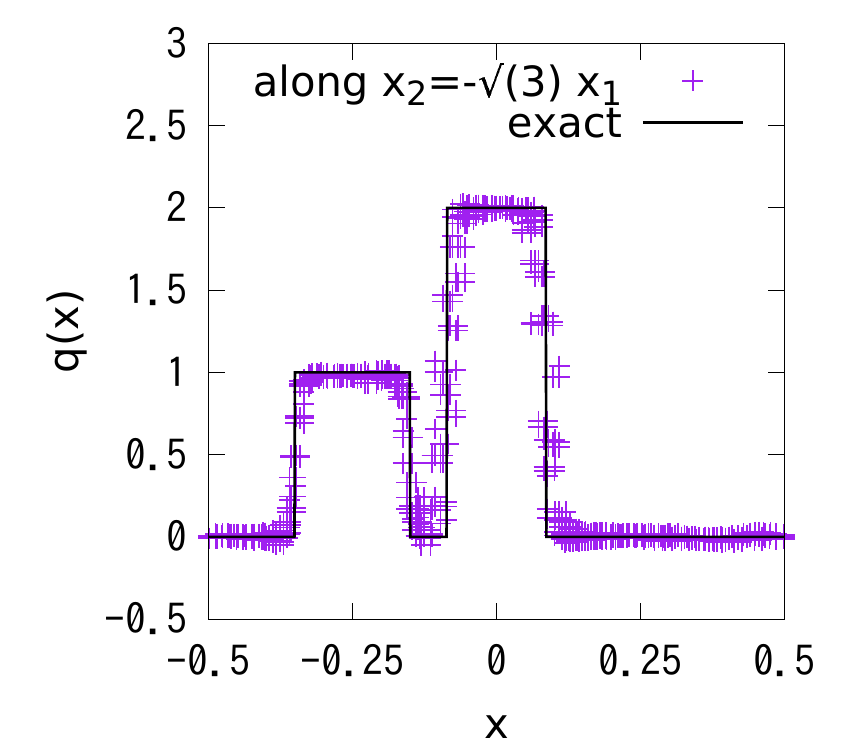}
\end{minipage}
\begin{minipage}{.43\textwidth}
\centering
\includegraphics[width=\textwidth]{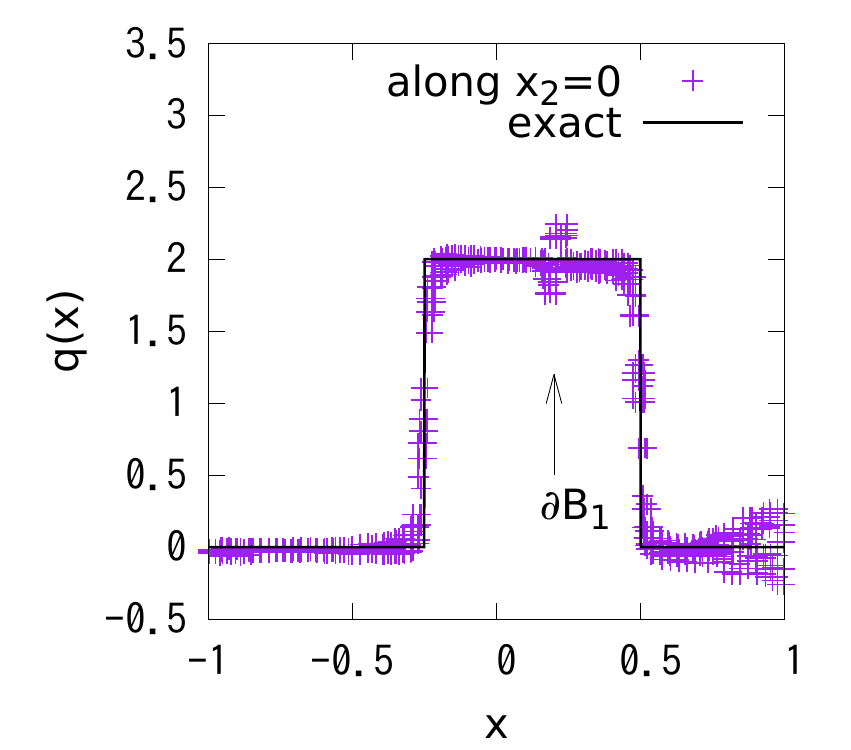}
\end{minipage}
\caption{\label{fig:disconta:reconst:section} 
Section of the reconstructed source along the dotted lines in \cref{fig:disconta:reconst},
 $|x_2+\sqrt{3}x_1| < 0.05$ (on the left) and $|x_2|<0.05$ (on the right) respectively.
The arrow on the right figure indicate $\partial B_1$ where the scattering coefficient is discontinuous.
}
\end{figure}

\begin{figure}[ht]
\centering
\includegraphics[width=.6\textwidth]{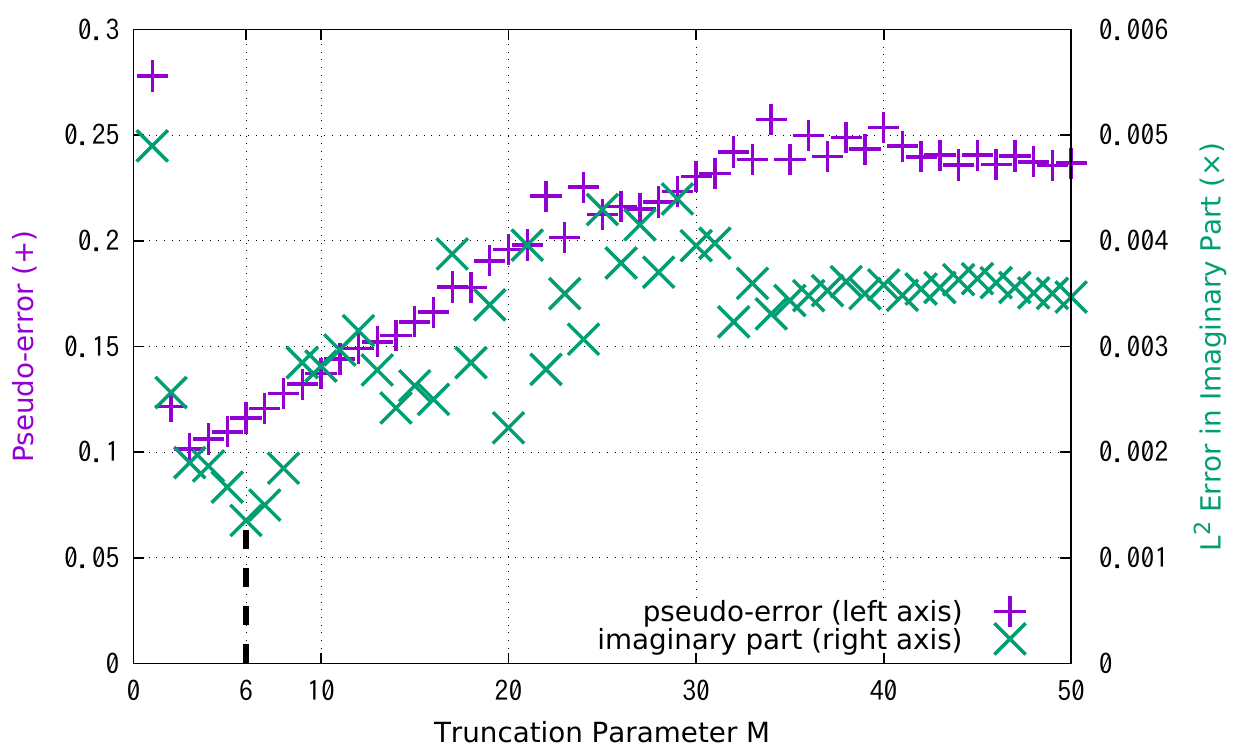}
\caption{\label{fig:example1:error}The truncation parameter $M$ and pseudo-errors in reconstructed source ($+$, left axis)
and corresponding imaginary part ($\times$, right axis) in Experiment 1.}
\end{figure}
\cref{fig:example1:error} shows
the relation between $M$, errors in the corresponding imaginary part \eqref{eq:error:imag} (computed independent of $q$),
and pseudo-errors (this require knowledge of $q$) in the reconstructed $q$ defined by
\begin{equation}\label{eq:error:real}
\left\{\sum_{\text{$\tau_{\ell}$ : $q$ is continuous}} |\tau_{\ell}| |q(z_{\ell}) - q_{\ell}|^2\right\}^{1/2},
\end{equation}
where $|\tau_{\ell}|$ is the area of $\tau_{\ell}$, $z_{\ell}$ is the center of the triangle $\tau_{\ell}$,
and the summation runs over the triangles where $q$ is continuous (in particular, constant in the example).
From the figures, the error in imaginary part scaled on the right axis
is sufficiently smaller the pseudo-error scaled on the left axis.
Both errors take minimum around $M=5$, and increase after that.
More precisely, the error of the imaginary part \eqref{eq:error:imag} is minimum at $M=6$,
while the pseudo-error of reconstructed $q$ is minimum at $M=3$.
According to the criteria stated before, we adopt $M=6$,
which causes $|p_{M}|< 2.49\times10^{-3}$ in the setting.

\end{experiment}

\begin{experiment}
Let us consider the situation that $\mua$ is given by the modified Shepp-Logan phantom~\cite{modifiedSL},
which occupies the ellipse
\[
\D = \left\{(x_1,x_2) \sep \dfrac{x_1^2}{a^2} + \dfrac{x_2^2}{b^2} < 1 \right\},
\quad a = 0.69, b = 0.92.
\]

\begin{figure}[ht]
\centering
\begin{tikzpicture}[scale=2.5]
\draw [color=gray,fill=gray] (-0.2,-0.705) rectangle (0.2,-0.505); 
  \node at (0.26,-0.605) {\scalebox{1}{$R$}};
\draw [color=gray,fill=gray] (-0.4,0) circle (0.1);
  \node at (-0.53,-0.15) {\scalebox{1}{$Q_1$}};
\draw [color=gray,fill=gray] (0.22,0) circle (0.05);
  \node at (0.35,-0.1) {\scalebox{1}{$Q_2$}};

\draw (0,0) ellipse (0.69 and 0.92); 
\draw (0,-0.0184) ellipse (0.6624 and 0.874); 
\draw [rotate=-18] (0.2092,0.06798) ellipse (0.11 and 0.31); 
\draw [rotate=18] (-0.2092,0.06798) ellipse (0.16 and 0.41); 
\draw (0,0.35) ellipse (0.21 and 0.25); 
\draw (0,0.1) ellipse (0.046 and 0.046); 
\draw (0,-0.1) ellipse (0.046 and 0.046); 
\draw (-0.08,-0.605) ellipse (0.046 and 0.023); 
\draw (0, -0.605) ellipse (0.023 and 0.023); 
\draw (0.06,-0.605) ellipse (0.023 and 0.046); 

\end{tikzpicture}
\caption{\label{fig:SLandQ} Discontinuous interface of $\mua$ (solid curves)
in the Shepp-Logan phantom, and the locations of support of the source $q$ (gray)}
\end{figure}
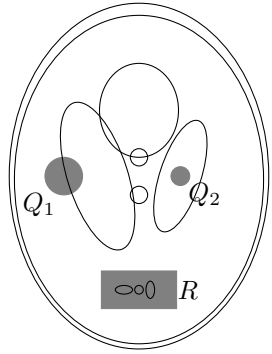

We consider three disjoint domains in $D$ (\cref{fig:SLandQ}):
\begin{align*}
Q_1 &= \set{(x_1+0.4)^2 + x_2^2 < 0.1^2},\\
Q_2 &= \set{(x_1-0.22)^2 + x_2^2 < 0.05^2},\\
R   &= \set{-0.2 < x_1 < 0.2, -0.705 < x_2 < -0.505},
\end{align*}
and the internal source is given by
\[
q(x) = \begin{cases}
2, \qquad & \text{in $Q_1\cup Q_2$};\\
1, \qquad & \text{in $R$};\\
0, \qquad & \text{otherwise}.
\end{cases}
\]

\begin{figure}[ht]
\begin{minipage}{.35\textwidth}
\includegraphics[width=\textwidth]{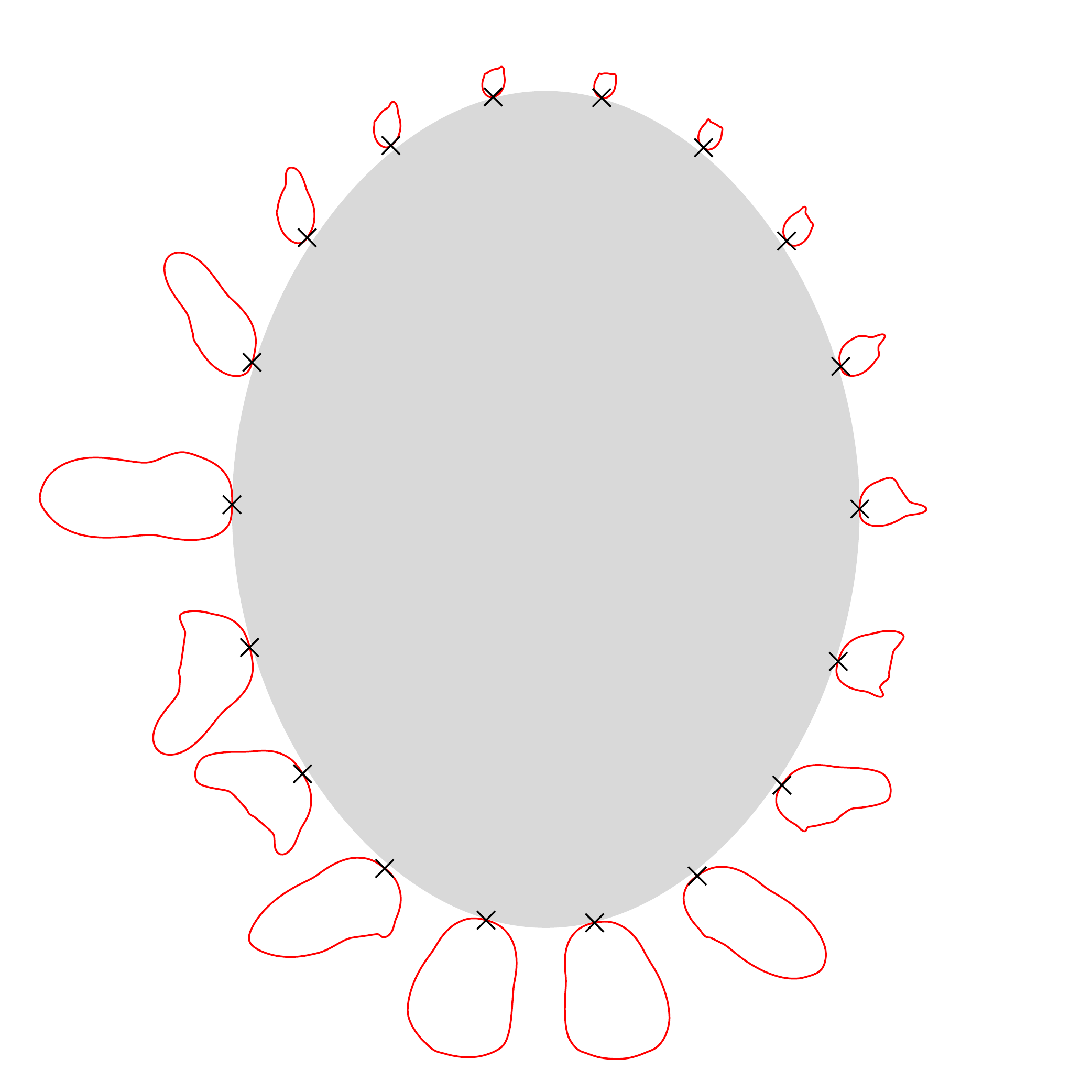}
\end{minipage}
\hfill
\begin{minipage}{.58\textwidth}
\includegraphics[width=\textwidth]{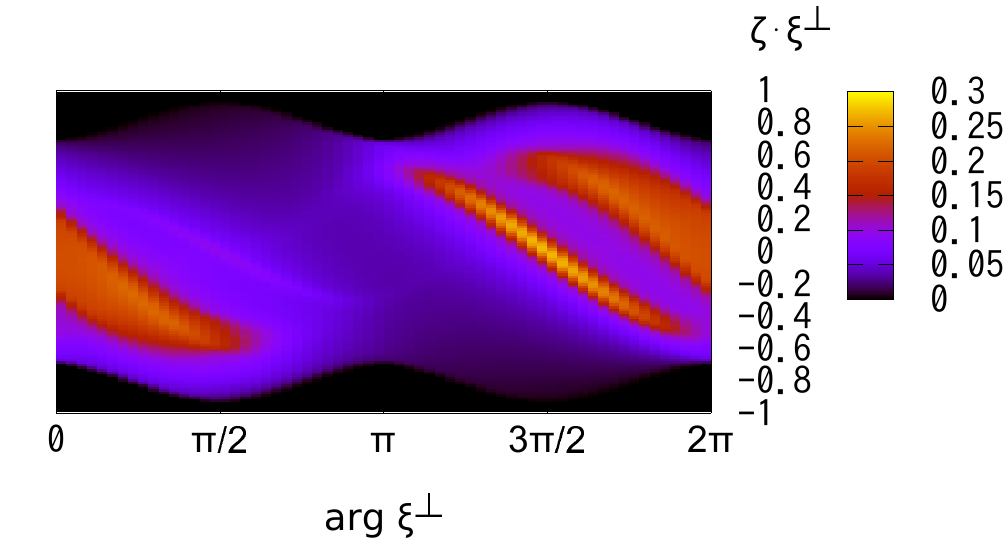}
\end{minipage}
\caption{\label{fig:SL-bd} Measurement data on $\partial\D$ is shown by red curves,
and the ellipse domain $D$ is filled by gray (on the left),
projection of measurement data (on the right)}
\end{figure}
The boundary measurement is generated by solving the forward problem
with $1,554,282$ triangles and $360$ velocity directions.
The numbers of measurement points are $K=3,000$ on $\partial\D$, and $N=360$ on $S^1$.
On the contrary, the reconstruction mesh consists of $6,010$ triangles with $3,106$ vertices.
Similarly as the previous experiment,
the latter mesh is generated without any information of $\mua$ and $q$.
The boundary nodes on the ellipse generated by FreeFem++ are not
equi-spaced with respect to their angles.

\begin{figure}[ht]
\centering
\begin{minipage}{.58\textwidth}
\includegraphics[width=\textwidth]{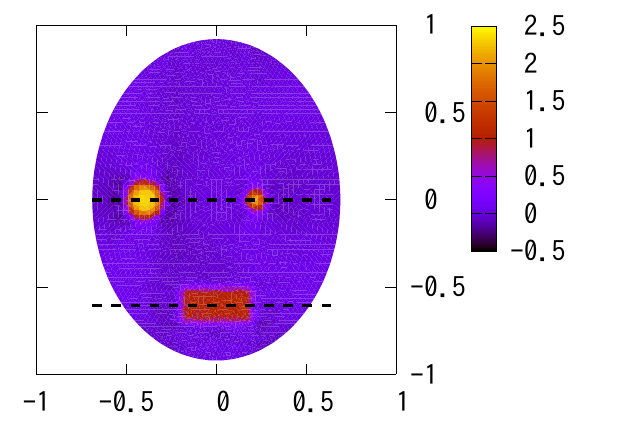}
\end{minipage}
\caption{\label{fig:SL}Numerical reconstruction of the source $q(x)$ (real part) on $\D$ with $M=8$ in Experiment 2.}
\end{figure}
\begin{figure}[ht]
\begin{minipage}{.48\textwidth}
\includegraphics[width=\textwidth]{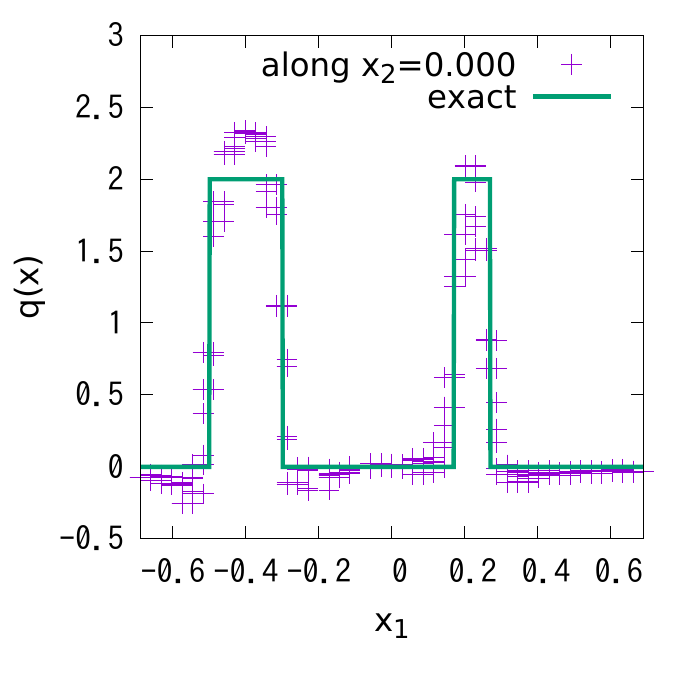}
\end{minipage}
\hfill
\begin{minipage}{.48\textwidth}
\includegraphics[width=\textwidth]{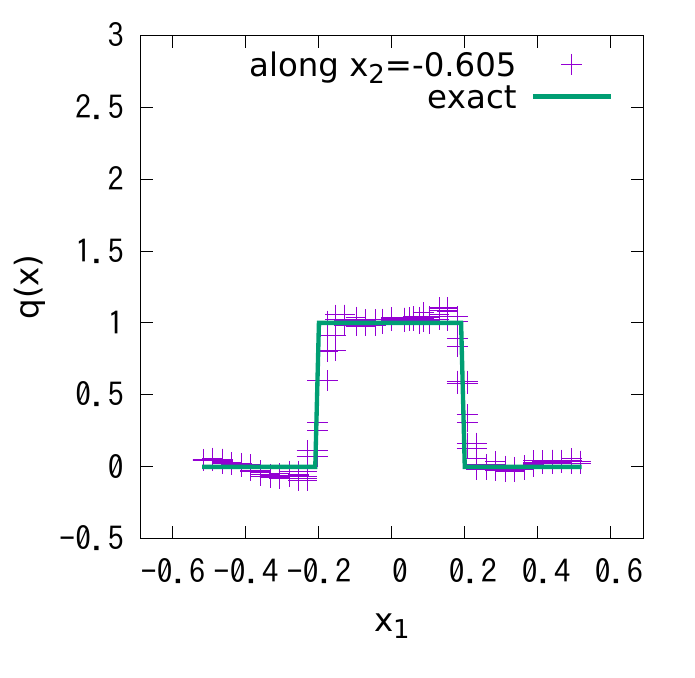}
\end{minipage}
\caption{\label{fig:SL:sectM8}Cross sections of the reconstructed source $q(x)$ with $M=8$
on the dotted lines in \cref{fig:SL},
$|x_2| < 0.01$ (on left) and $|x_2+0.605|< 0.01$ (on right).}
\end{figure}
\cref{fig:SL} depicts the reconstructed $q(x)$ on $\D$ with $M=8$,
while \cref{fig:SL:sectM8} is its sections on the dotted lines.
The computational time for reconstruction with $M=8$ is $247$ seconds
on Xeon E5-2650 v4 (2.2GHz, 12 cores) with OpenMP.
From the results, the support of $q$ is clearly and quantitatively reconstructed,
while the profile of $\mua$ does not appear.

\begin{figure}[ht]
\centering
\includegraphics[width=.6\textwidth]{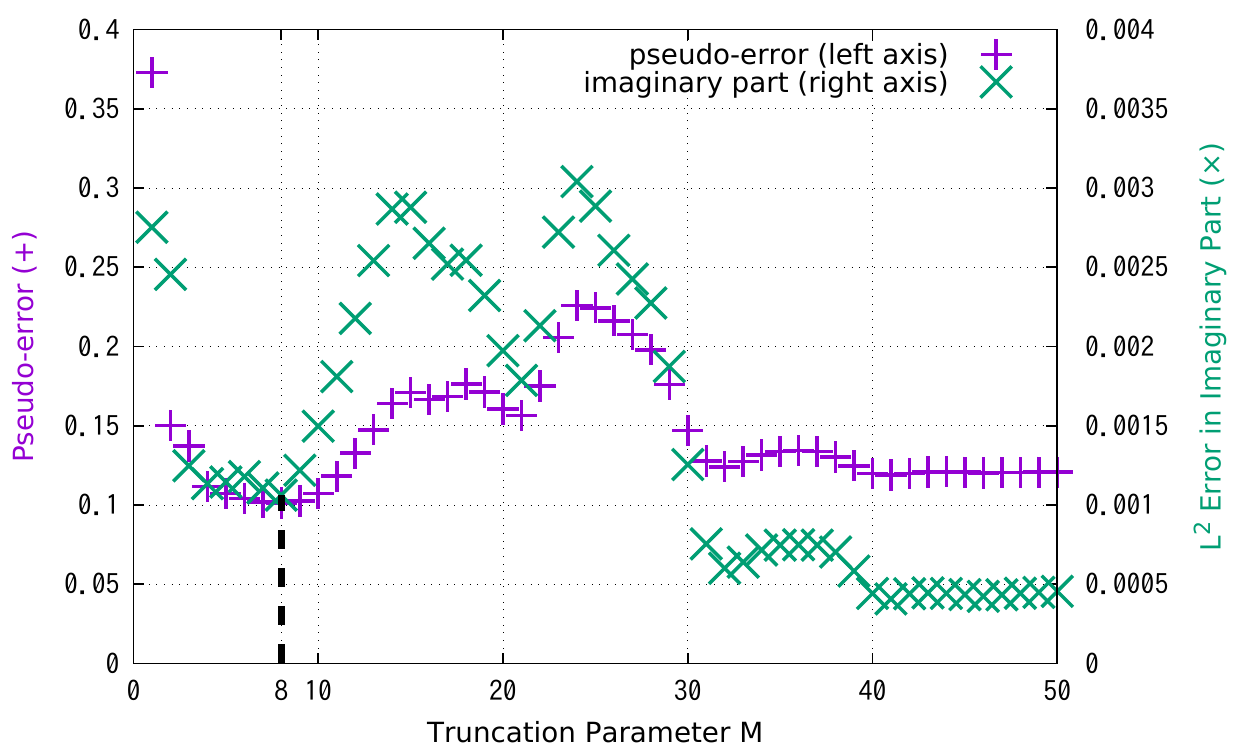}
\caption{\label{fig:SL:error}The truncation parameter $M$ and pseudo-errors
in reconstructed source ($+$, left axis) and corresponding imaginary part ($\times$, right axis).}
\end{figure}
The pseudo-errors in the reconstructed source and
the errors of the corresponding imaginary part are shown in \cref{fig:SL:error}.
Our proposed optimality criterion yields to choose $M=8$ and $M=40$ as reasonable orders of truncation.

\begin{figure}[ht]
\begin{minipage}{.48\textwidth}
\includegraphics[width=\textwidth]{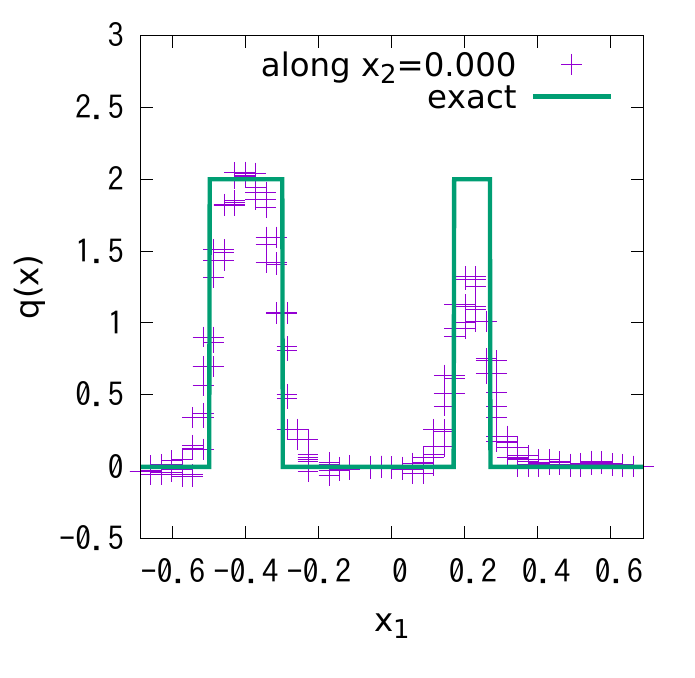}
\end{minipage}
\hfill
\begin{minipage}{.48\textwidth}
\includegraphics[width=\textwidth]{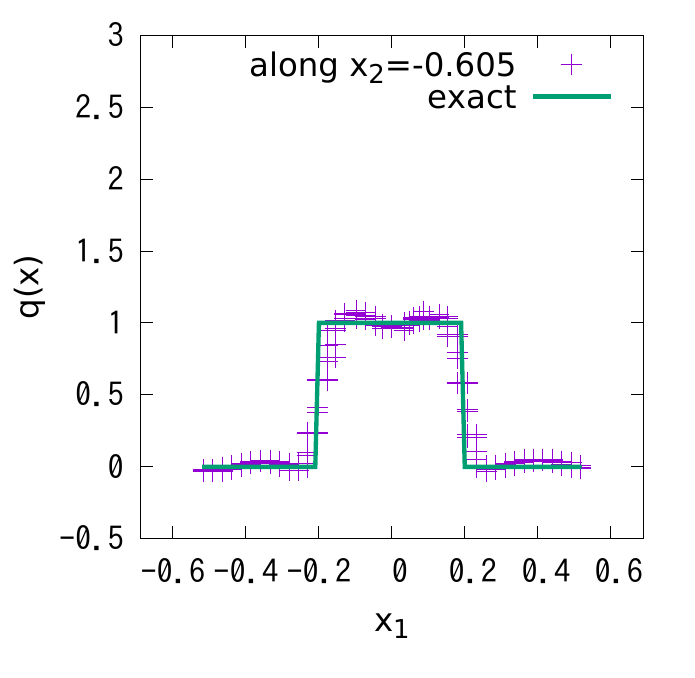}
\end{minipage}
\caption{\label{fig:SL:sectM40}Cross sections of the reconstructed source $q(x)$ with $M=40$
on the dotted lines in \cref{fig:SL},
$|x_2| < 0.01$ (on left) and $|x_2+0.605|< 0.01$ (on right).}
\end{figure}
\cref{fig:SL:sectM40} shows sections of the reconstructed $q(x)$
with $M=40$ on the dotted lines in \cref{fig:SL}.
Although the value of \eqref{eq:error:imag} at $M=40$ is smaller than that at $M=8$, 
the peak of the source in $Q_2$ is not obtained well with $M=40$.
The possible reason for this fact is that the size of $Q_2$
is small relative to the size of the triangular mesh, yielding that
the discrete $L^2$-norm used in computing the pseudo-error
in \eqref{eq:error:real}  be less effective.

\end{experiment}
In general, the choice of an optimal truncation parameter $M$ is not clear.
However, our algorithm includes an optimality criterion
which is independent of the knowledge of the source,
thus making it feasible. The numerical experiments
based on this choice were shown to produce accurate reconstructions.


\section*{Acknowledgment}
The work of H.~Fujiwara was supported by JSPS KAKENHI Grant Numbers 16H02155, 18K18719, and 18K03436.
The work of K.~Sadiq    was supported by the Austrian Science Fund (FWF), Project P31053-N32. 
The work of A.~Tamasan  was supported in part by the NSF grant DMS-1907097.
 

\end{document}